\documentclass[a4paper]{amsart}
\usepackage[textwidth=15cm,top=3cm, bottom=3cm, hcentering]{geometry}
\usepackage[colorlinks=true,linkcolor=red,citecolor=blue]{hyperref}

\usepackage{amsmath,amscd,amsthm,amssymb}
\usepackage{color}
\usepackage[utf8]{inputenc}
\normalfont
\usepackage[T1]{fontenc}
\usepackage{tikz-cd}
\usepackage{tikz}
\usepackage{enumerate}
 \usepackage[all]{xy}

\newtheorem{theo}{Theorem}[section]
\newtheorem{lemm}[theo]{Lemma}
\newtheorem{prop}[theo]{Proposition}
\newtheorem{coro}[theo]{Corollary}

\theoremstyle{definition}
\newtheorem{defi}[theo]{Definition}
\newtheorem{rema}[theo]{Remark}

\newtheorem{exam}[theo]{Example}

\theoremstyle{theorem}

\newcommand{\CC}{\mathbb{C}}

\newcommand{\KK}{\mathbb{K}}

\newcommand{\RR}{\mathbb{R}}

\newcommand{\ZZ}{\mathbb{Z}}
\newcommand{\Aa}{\mathcal{A}}

\newcommand{\Dd}{\mathcal{D}}

\newcommand{\Hh}{\mathcal{H}}

\newcommand{\Ll}{\mathcal{L}}
\newcommand{\Mm}{\mathcal{M}}

\newcommand{\Oo}{\mathcal{O}}

\newcommand{\Ho}{\mathrm{Ho}}
\newcommand{\Ker}{\mathrm{Ker}}

\newcommand{\Img}{\mathrm{Im}}

\newcommand{\ov}{\overline}

\newcommand{\del}{\partial}
\newcommand{\delb}{{\overline\partial}}

\newcommand{\lra}{\longrightarrow}

\newcommand{\ch}[1]{\mathrm{Ch^*({#1})}}

\newcommand{\MHS}{{\mathsf{MHS}}}
\newcommand{\MHC}{{\mathrm{MHC}}}
\newcommand{\BV}{{\mathrm{BV}}}
\newcommand{\Hycom}{{\mathrm{Hycom}}}
\newcommand{\Vect}{\mathrm{Vect}}
\newcommand{\Alg}{\mathrm{Alg}}

\newcommand{\Hom}{\mathrm{Hom}}
\newcommand{\Dol}{\mathrm{Dol}}
\newcommand{\dR}{\mathrm{dR}}
\newcommand{\Id}{\mathrm{Id}}

\newcommand{\vp}{\varphi}

\title[Formality of hypercommutative algebras of Kähler
and Calabi--Yau manifolds]{Formality of hypercommutative algebras of Kähler\\
and Calabi--Yau manifolds}

\author{Joana Cirici}
\author{Geoffroy Horel}

\address[J. Cirici]{Departament de Matemàtiques i Informàtica, Universitat de Barcelona\\
Gran Via 585\\
08007 Barcelona, Spain  / Centre de Recerca Matemàtica, Edifici C, Campus Bellaterra, 08193 Bellaterra, Spain}
\email{jcirici@ub.edu}

\address[G. Horel]{Université Sorbonne Paris Nord, Laboratoire Analyse, Géométrie et Applications, CNRS (UMR 7539), 93430, Villetaneuse, France.}
\email{horel@math.univ-paris13.fr}

\thanks{
J.~C.~acknowledges Govern de Catalunya (2021-SGR-00697 and Serra H\'{u}nter Program) and the Spanish State Research Agency (CEX2020-001084-M, EUR2023-143450, PID2020-117971GB-C22 and PID2024-155646NB-I00).
G. Horel acknowledges support from  a grant of the Romanian Ministry of Education and Research, CNCS - UE-
FISCDI, project number PN-III-P4-ID-PCE-2020-2798. Both authors are funded by the ANR-20-CE40-0016 HighAGT and the CNRS IEA00979.
}

\begin{document}

\begin{abstract}
Any Batalin--Vilkovisky algebra with a homotopy trivialization of the BV-operator gives rise to a hypercommutative algebra structure at the cochain level which, in general,
contains more homotopical information than the hypercommutative algebra introduced by Barannikov and Kontsevich on cohomology. In this paper, we use the purity of mixed Hodge structures to show that the canonical hypercommutative algebra defined on any compact Calabi--Yau manifold is formal. We also study related hypercommutative algebras associated to compact Kähler and Hermitian manifolds.
\end{abstract}

\maketitle

\section{Introduction}
Motivated by the mirror symmetry program,
Barannikov and Kontsevich defined a canonical hypercommutative algebra structure on the cohomology of any compact Calabi--Yau manifold \cite{BaKo}.
More generally, they gave a recipe for obtaining such a structure on the cohomology of any Batalin--Vilkovisky algebra satisfying the so-called \textit{$d\Delta$-condition}. This recipe, described in detail by Manin \cite{Maninbook}, \cite{Manin3},
arises as an abstract interpretation of the Tian--Todorov lemma leading to unobstructedness of the deformation theory of Calabi--Yau manifolds. It has been successfully applied to various geometric settings, such as symplectic manifolds satisfying Hard Lefschetz \cite{Merkulov} and compact Kähler manifolds \cite{CaZh}.

More recently, Drummond-Cole and Vallette \cite{DCV} refined the above recipe by obtaining a homotopy hypercommutative
structure on cohomology 
which allows one to recover the homotopy type of the original Batalin--Vilkovisky algebra.
Its truncation to a strict hypercommutative algebra is the one defined by Barannikov and Kontsevich via deformation theory. 
The scope of this refined theory is wider, in the sense that the  $d\Delta$-condition on the initial Batalin--Vilkovisky algebra is replaced by a weaker condition called \textit{Hodge-to-de-Rham degeneration}.
This turns out to be a sufficient condition for the existence of a
homotopy trivialization of the Batalin--Vilkovisky operator. The emergence of a hypercommutative algebra structure from this data
is a reflection of the fact that the homotopy quotient of the framed little discs operad by rotations
is weakly equivalent to the operad of moduli spaces of stable genus 0 curves \cite{DCRotations}.
In \cite{KMS}, Khoroshkin, Markarian and Shadrin proved an algebraic counterpart of this result, providing an explicit quasi-isomorphism between the operad  $\Hycom=\{H_*(\overline{\mathcal{M}}_{0,\bullet+1})\}$ governing hypercommutative algebras and the
homotopy quotient $\BV/\Delta$ governing 
Batalin--Vilkovisky algebras enhanced with a trivialization of the BV-operator. In particular, the hypercommutative structure associated to a Batalin--Vilkovisky algebra exists at the cochain level, rather than on cohomology.

As shown in \cite{DCV}, even under the stronger $d\Delta$-condition, the hypercommutative structure at the cochain level is not formal in general.
The authors ask whether, in the case of Calabi--Yau manifolds, the hypercommutative structure obtained by Barannikov and Kontsevich
on its cohomology, contains all the information of the homotopy type of the original Batalin--Vilkovisky algebra, thus giving formality.
Drummond-Cole gave an affirmative answer in the case of low-dimensional Calabi--Yau manifolds \cite{DC}.
In this paper, we give a positive answer in the general case of arbitrary dimension.
\medskip 

The standard proof of the well-known Formality Theorem for the de Rham algebra of compact Kähler manifolds \cite{DGMS}
uses the fact that the $\del\delb$-condition gives quasi-isomorphisms of commutative algebras 
\[(A,\delb,\wedge)\hookleftarrow (\Ker(\del),\delb,\wedge)\twoheadrightarrow(H,0,\wedge).\]
In the case of a Batalin--Vilkovisky algebra $(A,d,\cdot,\Delta)$ satisfying the $d\Delta$-condition one is led to consider an analogous string of quasi-isomorphisms
\[(A,d,\cdot,\Delta)\hookleftarrow (\Ker(\Delta),d,\cdot,0)\twoheadrightarrow(H,0,\cdot,0).\]
However, since $\Ker(\Delta)$ is not a commutative algebra with respect to the commutative product, the above maps do not lead to formality of the $\BV$-algebra, but only of
the associated Lie algebra, whose Lie bracket measures the failure of $\Delta$ to be a derivation. 
The formality of such a Lie algebra immediately leads to the unobstructedness of its associated deformation theory.

A more abstract approach to the formality of Kähler manifolds uses the fact that their cohomology carries pure Hodge structures, and that the filtrations defining such Hodge structures exist at the cochain level. 
More generally, the theory of mixed Hodge structures also leads to formality results in certain situations. 
Both the operad of Batalin--Vilkovisky algebras and the hypercommutative operad may be described as homologies of geometric objects and as such, they carry natural mixed Hodge structures satisfying certain purity properties. Moreover, the quasi-isomorphism of operads
$\Hycom\lra \BV/\Delta$, relating hypercommutative algebras with Batalin--Vilkovisky algebras with a trivialization of the BV-operator, is compatible with such mixed Hodge structures.
It therefore makes sense to talk about hypercommutative algebras in mixed Hodge structures and, more generally, in mixed Hodge complexes.
We actually consider $\CC$-mixed Hodge objects, for which there is no real structure and there are two (possibly unrelated) filtrations $F$ and $\overline{F}$ that are $n$-opposed in each $n$-graded piece of the weight filtration.
In this paper, we adapt the ``purity implies formality'' machinery of \cite{CiHo} to show that any hypercommutative algebra in $\CC$-mixed Hodge complexes whose Hodge structure is pure in cohomology, is formal as a hypercommutative algebra in complexes of vector spaces.
We study three geometric instances where this situation applies:

\subsection*{I. Kähler manifolds}
Consider the complex de Rham algebra
$\Aa:=\Aa^*_{\mathrm{dR}}(M)\otimes_\RR\CC$  of a complex manifold $M$. Given a Hermitian metric one obtains a decomposition $d^*=\delb^*+\del^*$ of the formal adjoint to $d=\delb+\del$.
Note that $\delb^*$ and $\del^*$ are not derivations in general. Instead, as shown by Cao--Zhou in \cite{CaZh} and \cite{CaZh2}, in the case of a Kähler manifold,
the tuples 
\[\Aa_{\mathrm{Dol}}:=(\Aa,d=\delb,\wedge,\Delta=-i\del^*)\text{ and }\Aa_{\mathrm{dR}}:=(\Aa,d=\delb+\del,\wedge,\Delta=i(\delb^*-\del^*))\] are $\BV$-algebras which satisfy the $d\Delta$-condition. Furthermore, there is a canonical trivialization of the operator $\Delta$, given in terms of the Green operator.
This allows one to talk about canonically induced hypercommutative structures in cohomology.

The following result promotes the Formality Theorem of \cite{DGMS} of compact Kähler manifolds to the hypercommutative setting:

\newtheorem*{I1}{\normalfont\bfseries Theorem \ref{chineseformality}}
\begin{I1}
For any compact Kähler manifold, the canonical hypercommutative structures defined on $\Aa_{\Dol}$ and $\Aa_{\dR}$
are formal and quasi-isomorphic.
\end{I1}

In particular, the canonically induced hypercommutative algebra structures on Dolbeault and complex de Rham cohomology of any compact Kähler manifold are isomorphic and carry no higher homotopy hypercommutative operations.

\subsection*{II. Calabi--Yau manifolds}
Let now $M$ be a compact Calabi--Yau manifold of complex dimension $m$ and consider the
complex bigraded algebra given by 
\[\Ll^{p,q}:=\Gamma(M,\Lambda^p T\otimes \Lambda^q\ov{T}^*),\]
where $T$ denotes the holomorphic tangent bundle of $M$. The algebra structure is determined by 
the exterior product $\wedge$.
There is an isomorphism of bigraded vector spaces $\eta:\Ll^{p,q}\to \Aa^{m-p,q}$ given by the contraction $\alpha\mapsto \alpha \vdash \Omega$
with the nowhere vanishing holomorphic $m$-form $\Omega$. Note that such an isomorphism does not preserve the product structures. In fact, the exterior product on $\Aa$ has different bidegree. Consider on $\Ll$ the operators
$\delb_\eta:=\eta^{-1}\circ \delb\circ \eta$ and $\del_\eta:=\eta^{-1}\circ \del\circ \eta$.
Barannikov and Kontsevich showed in \cite{BaKo} that the tuple 
\[\Ll_{\Dol}:=(\Ll,d=\delb_\eta,\wedge, \Delta=\del_\eta)\] is a Batalin--Vilkovisky algebra satisfying the $d\Delta$-condition.
As conjectured by Cao--Zhou in \cite{CaZh}, this algebra should correspond, via mirror symmetry, to the
algebra $\Aa_{\mathrm{Dol}}$ introduced earlier. With this idea in mind, and with the objective to obtain an object in mixed Hodge complexes, we introduce a ``de Rham version'' of this Batalin--Vilkovisky algebra
\[\Ll_{\dR}:=(\Ll^{*,*},d:=\delb_\eta+\del^*_\eta,\wedge,\Delta:=\del_\eta-\delb^*_\eta)\]
which, conjecturally, corresponds to $\Aa_{\dR}$ through mirror symmetry. The incorporation of the extra terms in both the differential and the BV-operator 
allows for the spectral sequences associated to both the column and row filtrations to degenerate at the first page. This is a necessary condition for obtaining a hypercommutative algebra in $\CC$-mixed Hodge complexes.
We show:
\newtheorem*{I2}{\normalfont\bfseries Theorem \ref{mainthm}}
\begin{I2}
For any compact Calabi--Yau manifold, the canonical hypercommutative structures defined on $\Ll_{\Dol}$ and $\Ll_{\dR}$ are formal and quasi-isomorphic.
\end{I2}

\subsection*{III. Hermitian manifolds}
A third geometric situation arises by considering, on a complex manifold, the tuple $(\Aa,\delb,\wedge,\del)$ as a $\BV$-algebra. In this case, $\del$ is obviously a derivation and so this is a $\BV$-algebra whose associated Lie bracket is trivial.
In the compact Kähler case, the $\del\delb$-condition immediately implies that such a $\BV$-algebra is formal and the induced hypercommutative structure is trivial. The interesting situation is therefore outside the Kähler setting. We show that for a compact Hermitian manifold, the induced hypercommutative algebra structure is not formal in general, and contains more information than the purely commutative structure. We do this by computing higher operations on the Kodaira-Thurston manifold.

\medskip 

The paper is organized as follows. In Section \ref{secBV} we collect preliminaries on Batalin--Vilkovisky and hypercommutative algebras. In particular, we recall, following \cite{KMS}, the procedure for obtaining a hypercommutative algebra structure at the cochain level, from a Batalin--Vilkovisky algebra with a trivialization of $\Delta$. Section \ref{secMHS} concerns mixed Hodge structures and formality. We first adapt the ``purity implies formality'' result of \cite{CiHo} to the setting of complex mixed Hodge structures (rather than real mixed Hodge structures) and then prove a general result of formality for operadic algebras with pure cohomology whose governing operad is also pure.
Lastly, in Section \ref{secGeom} we study the three main geometric situations mentioned above.

\subsection*{Acknowledgments}
We would like to thank Vladimir Dotsenko and Fernando Muro for answering our questions and Gabriel Drummond-Cole for introducing us to the problem that led to this work.

This project started during the authors's research stay at the Institut Mittag-Leffler, within the program 
\textit{Higher algebraic structures in algebra, topology and geometry} and continued during the research stay of the second author at the Institut de Matemàtiques de la Universitat de Barcelona (IMUB). We thank both institutes for the great working conditions and hospitality.

\section{From Batalin--Vilkovisky to hypercommutative algebras}\label{secBV}
In this preliminary section, we collect some results on double complexes, Batalin--Vilkovisky algebras and hypercommutative algebras.

\subsection{Double complexes and the $d\Delta$-condition}
Let $(A,d,\Delta)$ be a double complex of vector spaces, with $d$ and $\Delta$ linear operators of degree $+1$ and $-1$ respectively and satisfying 
\[d^2=\Delta^2=d\Delta+\Delta d=0.\]
Note for the general set-up in this section, double complexes will be $\ZZ$-graded. These are also called \textit{mixed complexes} in the literature. In the geometric situations we consider, double complexes will always carry a $\ZZ^2$-grading.

\begin{defi}
The  \textit{$d\Delta$-condition} is satisfied if and only if
\[\Ker(d)\cap \Img(\Delta)=\Ker(\Delta)\cap\Img(d)=\Img(d\Delta).\]
\end{defi}

Associated to a double complex there are various cohomologies:
aside from the $d$- and $\Delta$-cohomologies
\[H_d:={{\Ker(d)}\over{\Img(d)}}\quad\text{ and }\quad H_\Delta:={{\Ker(\Delta)}\over{\Img(\Delta)}}\]
we also have the Bott-Chern $H_{BC}$ and Aeppli $H_A$ cohomologies, given respectively by
\[H_{BC}:={{\Ker(d)\cap \Ker(\Delta)}\over{\Img(d\Delta) }}\quad\text{ and }\quad
H_{A}:={{\Ker(d \Delta)}\over{\Img(d)+\Img(\Delta)}}.\]
The identity induces maps 
\[\xymatrix{
H_{BC}\ar[d]\ar[r]&H_\Delta\ar[d]\\
H_d\ar[r]&H_A
}\]
which are isomorphisms if and only if the $d\Delta$-condition is satisfied (see Lemma 5.15 of \cite{DGMS}). 
In particular, in this case there is ``only one cohomology'', which we just denote by $H$.
The following is straightforward:
\begin{lemm}\label{quises}Let $(A,d,\Delta)$ be a double complex satisfying the $d\Delta$-condition. Then:
  \begin{enumerate}[(i)]
  \item The inclusions of complexes $(\Ker(\Delta),d)\hookrightarrow (A,d)$ and $(\Ker(d),\Delta)\hookrightarrow (A,\Delta)$ are quasi-isomorphisms.
   \item The surjections of complexes $(\Ker(\Delta),d)\twoheadrightarrow (H,0)$ and $(\Ker(d),\Delta)\twoheadrightarrow (H,0)$
are quasi-isomorphisms.
 \end{enumerate}
\end{lemm}

\begin{defi} A \textit{deformation retract} for a cochain complex $(A,d)$ is given by a diagram
\[
\begin{tikzcd}[ampersand replacement = \&]
(A,d) \arrow[r, rightarrow, shift left, "\rho"] \arrow[r, shift right, leftarrow, "\iota" swap] \arrow[loop left, "h"] \& (H_d,0).
\end{tikzcd}
\]
where $\iota$ and $\rho$ are morphisms of complexes
such that 
$\rho \iota=\mathrm{Id}$
and $h$ is a homotopy of morphisms of cochain complexes from $\iota\rho$ to the identity, so that $dh+hd=\iota\rho-\Id$.
\end{defi}

\begin{lemm}\label{HTData}Let $(A,d,\Delta)$ be a double complex satisfying the $d\Delta$-condition. Then
there is a deformation retract $(\iota,\rho,h)$
such that $\Delta \iota=0$, $\rho\Delta=0$ and $h\Delta+\Delta h=0.$
\end{lemm}
\begin{proof}
By  \cite[Proposition 3.17]{DGMS}, the
 $d\Delta$-condition allows one to write $A$
in each degree as a direct sum decomposition 
\[A^{n}=H^{n}\oplus S^{n}\oplus dS^{n-1}\oplus \Delta S^{n+1}\oplus d\Delta S^{n},\]
where $H^{n}$ is the cohomology of $A$ in degree $n$ and the differentials $d$ and $\Delta$ are given by
\[d(x,y,dz,\Delta w,d\Delta t)=(0,0,dy,0,d\Delta w)\text{ and }\Delta(x,y,dz,\Delta w,d\Delta t)=(0,0,0,\Delta y,-d\Delta z).\]
Define $\iota:H\to A$ and $\rho:A\to H$ respectively by inclusion and projection to the first component. 
Define $h:A^{p,q}\to A^{p,q-1}$ by 
$h(x,y,dz,\Delta w,d\Delta t):=-(0,z,0,\Delta t,0).$
\end{proof}

\subsection{Batalin--Vilkovisky algebras}A Batalin--Vilkovisky algebra is a Gerstenhaber algebra together with a compatible unary operation squaring to zero.
While Gerstenhaber algebras may be identified with algebras over the homology of the little disks operad, the operad $\BV$ encoding
Batalin--Vilkovisky algebras is given by the homology of the framed little disks operad \cite{Getzler2D}.

\begin{defi}
 A \textit{$\BV$-algebra structure} on a cochain complex $(A,d)$ is given by a symmetric binary product
 $\cdot:A\times A\lra A$ of degree 0,
 together with a unary operator $\Delta:A\to A$ of degree -1, such that:
 \begin{enumerate}[(1)]
 \item The triple $(A,d,\Delta)$ is a double complex: \[d^2=d\Delta+\Delta d=\Delta^2=0.\]
 \item The triple $(A,d,\cdot)$ is a dg-commutative algebra: 
 \[x\cdot y=(-1)^{|x|\cdot|y|}y\cdot x\quad\text{ and }\quad d(x\cdot y)=dx\cdot y+(-1)^{|x|}x\cdot dy.\]
 \item The binary operation $[-,-]:A\times A\lra A$ defined by 
  \[[x, y]:=(-1)^{|x|}(\Delta(x\cdot y)-(\Delta(x)\cdot y)-(-1)^{|x|} x\cdot \Delta(y))\]
  defines, for all $x$, a derivation $[x, -]$ of degree $|x|-1$ with respect to the commutative product:
   \[[x,y\cdot z]=[x, y]\cdot z+(-1)^{(|x|-1)\cdot |y|}y\cdot[x, z].\]
\end{enumerate}
  \end{defi}

\begin{rema}
Note that the last condition is equivalent to imposing the relation
\[\arraycolsep=4pt\def\arraystretch{1.4}
\begin{array}{ll}
 \Delta(x\cdot y\cdot z)=&\Delta(x\cdot y)\cdot z+(-1)^{|x|}x\cdot \Delta(y\cdot z)+(-1)^{(|x|-1)\cdot |y|}y\cdot \Delta(x\cdot z)-\\&-\Delta(x)\cdot y\cdot z -(-1)^{|x|}x\cdot \Delta(y)\cdot z-(-1)^{|x|+|y|}x\cdot y\cdot \Delta(z)\end{array}\]
 appearing in \cite{Getzler2D} to define a Batalin--Vilkovisky algebra. In particular, by Proposition 1.2 of \cite{Getzler2D},
given a BV-algebra $(A,d,\cdot,\Delta)$ with associated bracket $[-,-]$, then 
$(A,d,\cdot,[-,-])$ is a Gerstenhaber algebra and so $(A,d,[-,-])$ is a dg-Lie algebra. 
\end{rema}

\begin{rema}
The $d\Delta$-condition implies in particular that
the associated dg-Lie algebra $(A, [-,-], d)$ is both formal and quasi-abelian. 
Indeed, we have quasi-isomorphisms of dg-Lie algebras
\[(A,[-,-],d)\stackrel{\sim}{\longleftarrow}(\Ker(\Delta),[-,-],d)\stackrel{\sim}{\longrightarrow}(H,0,0).\]
Note however that 
the above are not maps of $\BV$-algebras: the Lie bracket is an obstruction for $\Delta$ to be a derivation, and so
in general $(\Ker(\Delta),\cdot)$ is not a commutative algebra.
Therefore these maps do not give formality of the $\BV$-algebra.
\end{rema}

\subsection{Hypercommutative algebras}\label{BVtoHycom}
Following \cite{KMS}, we review the equivalence between the category of Batalin--Vilkovisky algebras enhanced with the trivialization of $\Delta$  and the category of hypercommutative algebras.

\begin{defi}
A \textit{hypercommutative algebra} is a cochain complex $(A,d)$ with a sequence of
graded
symmetric $n$-ary operations
\[m_n:A^{\otimes n}\to A\] of degree $2(2-n)$, compatible with the differential and
satisfying the following generalized associativity condition:
for all $n\geq 0$ and $a,b,c,x_j\in A$, 
\[\sum_{S_1\sqcup S_2=\{1,\cdots,n\}} \pm m_*(m_*(a,b,x_{S_1}),c,x_{S_2})=
\sum_{S_1\sqcup S_2=\{1,\cdots,n\}} \pm m_*(a,m_*(b,c,x_{S_1}),x_{S_2}),
\]
where $\pm$ is the Koszul sign rule and $x_S$ denotes $x_{s_1},\cdots,x_{s_m}$ for 
$S=\{s_1,\cdots,s_m\}$.
\end{defi}

\begin{rema}
We may think of the operations in a hypercommutative algebra as the Taylor coefficients
of a formal deformation of the commutative product: the above relations are equivalent
to asking that the binary operation 
\[(a,b)\mapsto \sum {\hbar^n\over n!} m_{n+2}(a,b,x,\cdots,x)\in A[[\hbar]]\]
is associative for any $x$ (see \cite{Getzler}).
\end{rema}

The operad 
$\Hycom$ encoding hypercommutative algebras is a cyclic operad which may be identified with the homology  $H_*(\ov{\Mm}_{0,\bullet+1})$
of the moduli spaces of stable genus zero curves and marked points \cite{Getzler}, \cite{KoMa}. With this identification,
the arity $n$ generator in $\mathrm{Hycom}$  corresponds to the fundamental 
class \[[\ov{\Mm}_{0,n+1}]\in H_{2(n-2)}(\ov{\Mm}_{0,n+1}).\] 

In \cite{KMS}, Khoroshkin,  Markarian and Shadrin
give a quasi-isomorphism of operads \[\mathrm{Hycom}\lra \BV/\Delta,\] where 
$\BV/\Delta$ denotes the homotopy quotient of the operad $\BV$ by the generator $\Delta$.

The operad $\BV/\Delta$ is defined by adding, to $\BV$, generators $\vp_i$ of degree $-2i$, for $i\geq 1$ together with a differential $\delta$ defined as follows:
write $\vp(z):=\sum_{i\geq 1} \vp_i z^i$, where $z$ is a formal parameter. The expansion of the \textit{gauge equation}
\[
\exp(\vp(z)) d=(d+\Delta z)\exp(\vp(z))
\]
gives the various terms
\[\arraycolsep=4pt\def\arraystretch{1.4}
\begin{array}{l}
\Delta=[d,\vp_1],\\
0=[d,\vp_2]+{1\over 2}[[d,\vp_1],\vp_1],\\
0=[d,\vp_3]+[[d,\vp_1],\vp_2]+[[d,\vp_2],\vp_1]+{1\over 6}[[[d,\vp_1],\vp_1],\vp_1],\\
\cdots
\end{array}
\]
The differerential $\delta$ on $\vp_i$ is then defined by the equation $\delta \vp_i:=[d,\vp_i]$. The lower terms are:
\[
\arraycolsep=4pt\def\arraystretch{1.4}
\begin{array}{l}
\delta \vp_1= \Delta,\\
\delta\vp_2=-{1\over 2}[\Delta,\vp_1],\\
\delta\vp_3=-[\Delta,\vp_2]+{1\over 3}[[\Delta,\vp_1],\vp_1],\\
\cdots
\end{array}
\]
The map of operads $\mathrm{Hycom}\lra \BV/\Delta$ is defined by assigning, to each generator $m_n\in \Hycom(n)$, an element $\theta_i$ in $\BV/\Delta(n)$ represented by a finite sum of combinations of the arity $2$ product of $\BV$ together with generators $\vp_{j_1},\cdots,\vp_{j_k}$ with $j_1+\cdots+j_k=n-2$
(we refer to \cite[Section 1.2]{KMS} for details). 

\begin{exam}\label{theta3}
The lowest terms are
$\theta_2(x,y)=x\cdot y$ and 
\[
\arraycolsep=4pt\def\arraystretch{1.4}
\begin{array}{ll}
\theta_3(x,y,z)=&
\vp_1(x\cdot y\cdot z)+\vp_1(x)\cdot y\cdot z +x\cdot \vp_1(y)\cdot z +x\cdot y\cdot\vp_1(z)-\\
&-\vp_1(x\cdot y)\cdot z-(-1)^{|y|\cdot|z|}\vp_1(x\cdot z)\cdot y-x\cdot \vp_1(y\cdot z).
\end{array}
\]

\end{exam}

The above description gives a method for producing a hypercommutative algebra from any $\BV$-algebra $(A, d,\wedge,\Delta)$
enhanced with a homotopy that trivializes the BV-operator: assume there is a 
deformation retract $(\iota,\rho,h)$ for the complex $(A,d)$, satisfying the \textit{Hodge-to-de-Rham degeneration condition}
\[\rho(\Delta h)^{k-1}\Delta \iota=0\quad\text{ for all }\quad k\geq 1.\]
Then a solution
$\vp(z):=\sum_{n\geq 0}\vp_n z^n$ to the gauge equation is given by
\[\vp_n={(h\Delta)^n\over n}-n\sum_{\ell=1}^n {(h\Delta)^{\ell-1}\iota\rho (\Delta h)^{n-\ell+1}\over \ell}\]
(see Remark 2.2 of \cite{DSV}). In particular, we get the following proposition.

\begin{prop}\label{hycombigraded}
Let $(A, d,\wedge,\Delta)$ be a $\BV$-algebra together with a deformation retract $(\iota,\rho,h)$ for the complex $(A,d)$, satisfying
$\Delta \iota=\rho \Delta=h\Delta+\Delta h=0.$
Then $(A,d)$ carries a hypercommutative structure, where the operation $m_n$ of degree $4-2n$ is given by 
 a finite composition of the commutative product $\wedge$
together with copies of the unary operation $h\Delta$.
\end{prop}

\begin{proof}
It suffices to check that, in this case, we have $\varphi_1=h\Delta$ and $\varphi_n=0$ for all $n>1$.
\end{proof}

\section{Weights and formality}\label{secMHS}

We explain in this section how to obtain formality using the theory of weights.
We adapt the theory developed in \cite{CiHo} to the setting of complex mixed Hodge structures, 
and prove a result of formality for operadic algebras whose cohomology is pure.
\subsection{Complex mixed Hodge structures}

Classically, Hodge structures are defined on a vector space $V$ defined over a subfield of the real numbers.
Then, on the complexification $V_\CC:=V\otimes\CC$, there is a notion of complex conjugation
and having a pure Hodge structure of weight $n$ on $V$ is equivalent to having a filtration 
$F$ on $V_\CC$ which is $n$-opposed to its complex conjugate filtration.
We will instead be considering \textit{complex Hodge structures}. 
These are given by pairs of $n$-opposed filtrations on a complex vector space,
which are not necessarily complex conjugate to each other.

\begin{defi}Let $n$ be an integer.
Two decreasing filtrations $F$ and $\ov{F}$ on a vector space $V$ are said to be \textit{$n$-opposed} if 
\[Gr^p_FGr^q_{\ov{F}}V=0\text{ for all }p+q\neq n.\]
This is equivalent to asking that $F^pV\oplus\ov{F}{}^qV\cong V$ for all $p,q$ such that $p+q=n+1$.
\end{defi}

\begin{defi}A \textit{$\CC$-mixed Hodge structure} is a complex vector space $V$ with an
increasing filtration $W$ and two decreasing filtrations $F$ and $\ov{F}$ 
such that, for all $n$, the filtrations 
$F$ and $\ov{F}$ induced on $Gr_n^WV:=W_{n}V/W_{n-1}V$ are $n$-opposed:
\[Gr^p_FGr^q_{\ov F}Gr_n^W V=0\text{ for all }p+q\neq n.\]
\end{defi}
Denote by $\MHS_\CC$ the category of $\CC$-mixed Hodge structures. It is an abelian symmetric monoidal category.
Morphisms in this category, defined as morphisms of complex vector spaces compatible with all filtrations, 
are in fact strictly compatible with each filtration. This follows from Deligne's splitting lemma,
which is central in the study of mixed Hodge structures:

\begin{lemm}[Deligne's splitting lemma (\cite{DeHII}, Lemma 1.2.11)]\label{DelSplit}
Let $(V,W,F,\ov F)$ be a $\CC$-mixed Hodge structure. There are functorial decompositions
\[V=\bigoplus I^{p,q}_i\quad \text{ for }\quad i=0,1\]
 such that 
 \[F^pV=\bigoplus_{p'\geq p} I^{p',q'}_0,\quad\ov{F}{}^qV=\bigoplus_{q'\geq q} I^{p',q'}_1\quad\text{ and }\quad W_mV=\bigoplus_{m\leq p+q}I^{p,q}_i.\]
\end{lemm}

In order to obtain mixed Hodge structures on the cohomology of complex algebraic varieties, Deligne \cite{DeHIII} introduced \textit{mixed Hodge complexes} as
objects somewhat more flexible than complexes of mixed Hodge structures. The following definition is a shifted version of the corresponding notion 
in the setting of $\CC$-mixed Hodge structures.

\begin{defi}
 A \textit{$\CC$-mixed Hodge complex} is a cochain complex $(A,d)$ over $\CC$ with three filtrations $W$, $F$ and $\ov{F}$ satisfying:
\begin{enumerate}
 \item [$(H_0)$] The cohomology $H^*(A)$ is of finite type.
 \item [$(H_1)$] The bifiltered complexes $(A,W,F)$ and $(A,W,\ov F)$ are \textit{$d$-bistrict}: for $G=F$ or $\ov F$,
 the four spectral sequences
$$
\xymatrix@R=.0pc{
E_1(Gr^W_\bullet A,G)\ar@{=}[dd]\ar@{=>}[r]&E_1(A,W)\ar@{=>}[rd]\\
&&H(A)\\
E_1(Gr_G^\bullet A,W)\ar@{=>}[r]&E_1(A,G)\ar@{=>}[ru]\\
}
$$
degenerate at $E_1$.
 \item [$(H_2)$] For all $p$, the filtrations induced by $F$ and $\ov{F}$ on $H^n(Gr_p^WA)$ are $p$-opposed. 
\end{enumerate}
\end{defi}

Denote by $\MHC_\CC$ the category $\CC$-mixed Hodge complexes.
Morphisms are given by morphisms of cochain complexes compatible with all filtrations.
The cohomology of any $\CC$-mixed Hodge complex is a graded $\CC$-mixed Hodge structure. Also, there is an inclusion 
$\ch{\MHS_\CC}\hookrightarrow \MHC_\CC$.
Note that given a mixed Hodge complex in the sense of Deligne, one may always get a $\CC$-mixed Hodge 
complex by taking the décalage of the weight filtration on its complex part.

Beilinson gave an equivalence of categories between the derived category of mixed
Hodge structures and the homotopy category of mixed Hodge complexes. In \cite{CiHo} we gave a symmetric monoidal enhancement of this result in terms of $\infty$-categories. We prove an analogous result for $\CC$-mixed Hodge complexes.

\begin{prop}\label{equivalenceproof}
The inclusion  $\ch{\MHS_\CC}\hookrightarrow \MHC_\CC$ induces an equivalence of symmetric monoidal $\infty$-categories.
\end{prop}
\begin{proof}
Since both $\infty$-categories are stable and the inclusion is exact and symmetric monoidal, it suffices to prove that
the inclusion induces an equivalence of homotopy categories 
\[\Dd(\MHS_\CC)\stackrel{\sim}{\lra} \mathrm{Ho}(\MHC_\CC).\]
This follows analogously to Beilinson's equivalence for the case of real mixed Hodge structures (see \cite{Beilinson}, see also \cite{CG2}). The only minor difference is the proof that 
every $\CC$-mixed Hodge complex $(A,d,W,F,\ov F)$ is quasi-isomorphic to its cohomology, giving that the inclusion induces an essentially surjective functor at the homotopy level. We explain this below.

Since $H^n(A)$ is a $\CC$-mixed Hodge structure, for every $n\geq 0$ there are decompositions
 \[H^n(A)=\bigoplus H^{p,q}_i\text{ for }i=0,1\]
with the properties of Lemma \ref{DelSplit}.
Take sections $\sigma:H^{p,q}_0\to Z^n(A)$ and $\ov{\sigma}:H^{p,q}_1\to H^n(A)$.
Let $G=F$ or $\ov F$. Since the bifiltered complex $(A,d,W,G)$ is $d$-bistrict, we have
\[W_pG^qH^n(A)=\mathrm{Im}(H^n(W_pG^qA)\to H^n(A)).\]
Since $H^{p,q}_0\subseteq W_{p+q}F^qH^n(A)$ and $H^{p,q}_1\subseteq W_{p+q}\ov F^qH^n(A)$ we get
 a morphism of complexes $\sigma:H^n(A)\to A$ compatible with $W$ and $F$ and a morphism of complexes $\ov\sigma:H^n(A)\to A$ compatible with $W$ and $\ov F$.
Now, for any cohomology class $x\in H^n(A)$, the difference $(\sigma-\ov{\sigma})(x)$ is a coboundary. The strictness of $d$ with respect to $W$ ensures that if $x\in W_p$, then there is  $y\in W_p$ with 
$(\sigma-\ov{\sigma})(x)=dy$. Picking a basis of $H^*(A)$, the assignment $x\mapsto y$
for each basis element $x$ defines a linear map $\Sigma:H^*(A)\to H^*(A)[-1]$ which is a homotopy
 $\Sigma:\sigma\simeq \ov\sigma$ compatible with $W$.
We next strictify the triple $(\sigma,\ov\sigma,\Sigma)$ using the mapping cylinders of $\sigma$ and $\ov\sigma$.
The mapping cylinder of $\sigma$ carries compatible filtrations $W$ and $F$, 
given by 
\[W_mF^p\mathrm{Cyl}^n(\sigma):=W_mF^pH^{n+1}(A)\oplus W_mF^pA^n\oplus W_mF^pH^n(A)\]
with differential $d(x,a,x')=(0,da-\sigma x,0)$.
The mapping cylinder for $\ov\sigma$ is defined analogously, and carries filtrations $W$ and $\ov F$.
There is an isomorphism of complexes compatible with $W$
\[\varphi:\mathrm{Cyl}(\sigma)\to \mathrm{Cyl}(\ov\sigma)\quad\text{given by}\quad 
(x,a,x')\mapsto (x,a-\Sigma(x),x').\]
The filtration $\ov F$ defined on $\mathrm{Cyl}(\sigma)$ by 
\[\ov F{}^p\mathrm{Cyl}(\sigma):=\varphi^{-1}(\ov{F}{}^p\mathrm{Cyl}(\ov\sigma))\]
is compatible with the differential and so $(\mathrm{Cyl}(\sigma),d,W,F,\ov F)$ is a trifiltered complex.
The inclusions $i:H^*(A)\to \mathrm{Cyl}(\sigma)\leftarrow A:j$ are compatible with the three filtrations and induce isomorphisms in cohomology. In particular, the axioms $(H_0)-(H_2)$ for a $\CC$-mixed Hodge complex are verified by 
$\mathrm{Cyl}(\sigma)$ and so $A$ and $H^*(A)$ are isomorphic in $\Ho(\MHC_\CC)$.
\end{proof}

\subsection{Purity implies formality}

On the category of graded vector spaces over a field $\KK$, there are two distinct choices of braiding isomorphisms that make it into a symmetric monoidal category. In one of them, the isomorphism
\[A\otimes B\to B\otimes A\]
sends an elementary tensor of homogeneous elements $a\otimes b$ to $b\otimes a$. In the other $a\otimes b$ is sent to $(-1)^{|a||b|}b\otimes a$. We shall denote these symmetric monoidal structures $Gr\Vect_\KK$ and $gr\Vect_\KK$ respectively.
In the first category, the source for gradings will be the weights of mixed Hodge theory, while in the second category gradings are cohomological.

We recall the definition of formality for algebras over graded operads. There is an obvious symmetric monoidal functor
\[T:gr\Vect_\KK\to \ch{\KK}\]
giving a graded vector space the trivial differential. Let $\Oo$ be an operad in $gr\Vect_{\KK}$. We shall often omit $T$ in the notation and view a graded object as a chain complex via $T$. The functor $T$ induces a functor
\[T:\Alg_{\Oo}(gr\Vect_\KK)\to\Alg_{\Oo}(\ch{\KK})\]

\begin{defi}
An $\Oo$-algebra is said to be \textit{formal} if it is in the homotopical  essential image of the functor $T$.
\end{defi}

\begin{rema}
There are two ways of interpreting this definition depending on whether we consider the model category of $\Oo$-algebras or the $\infty$-category of $\Oo$-algebras. It turns out that these two ways are equivalent thanks to a rigidification theorem due to Haugseng \cite[Theorem 4.10]{haugsengalgebras}. The $\infty$-categorical interpretation will be more convenient to work with in the following.
\end{rema}

\begin{defi}Let $\alpha\neq 0$ be a rational number.
A $\CC$-mixed Hodge complex $(A,d,W,F,\ov{F})$ is said to be \textit{$\alpha$-pure} if
\[Gr_{p}^W H^n(A)=0\text{ for all }p\neq \alpha n.\]
\end{defi}
The full subcategory $\MHC_\CC^{\alpha\text{-pure}}$ of $\CC$-mixed Hodge complexes that are $\alpha$-pure is a symmetric monoidal category.
The main result of \cite{CiHo} gives formality for algebraic varieties whose cohomology is $\alpha$-pure, where $W$ is Deligne's weight filtration. 

We shall now adapt the theory to the case of interest to us. We consider the $\infty$-category $\ch{Gr\Vect_\KK}$ of cochain complexes of graded vector spaces with quasi-isomorphisms inverted. We denote by $\ch{Gr\Vect_\KK}^{\alpha\text{-pure}}$ the full subcategory whose objects are the chain complexes such that
\[H^n(C)_p=0 \text{ for all }p\neq \alpha n.\]
Clearly the symmetric monoidal structure on $\ch{Gr\Vect_\KK}$ restricts to one on $\ch{Gr\Vect_\KK}^{\alpha\text{-pure}}$.

\begin{prop}
Let $\alpha=r/s$ with $r$ and $s$ two coprime integers. The cohomology functor
\[\ch{Gr\Vect_\KK}^{\alpha\text{-pure}}\to gr\Vect_\KK\]
induces a symmetric monoidal equivalence of $\infty$-categories onto the full subcategory of $gr\Vect_\KK$ spanned by graded vector spaces that are zero except in degree divisible by $s$.
\end{prop}

\begin{proof}
This functor is clearly essentially surjective. In order to prove that it is fully faithful, it suffices to observe the following formula for the homotopy groups of the mapping spaces in $\ch{Gr\Vect_\KK}$
\[\pi_k\mathrm{Map}_{\ch{Gr\Vect_\KK}}(C,D)=\bigoplus_{i\in\mathbb{Z}}[C_i,D_i[k]]\]
in particular, we see that if $C$ and $D$ are both $\alpha$-pure, all the higher homotopy groups are trivial and we have 
\[\pi_0\mathrm{Map}_{\ch{Gr\Vect_\KK}}(C,D)=\bigoplus_{j\in\mathbb{Z}}\Hom(H^{sj}(C_{rj}),H^{sj}(D_{rj}))\]
This shows that the cohomology functor is fully faithful.
\end{proof}

\begin{theo}\label{purity_implies_formality}
Let $\Oo$ be an operad in graded mixed Hodge structures which is $\alpha$-pure. Every $\Oo$-algebra in $\alpha$-pure mixed Hodge complexes is formal as an algebra in $\ch{\CC}$.
\end{theo}

\begin{proof}
An immediate consequence of Deligne's splitting gives a homotopy commutative diagram of symmetric monoidal $\infty$-categories
\[\xymatrix{
\ch{\MHS_\CC}\ar[r]^-{\simeq}\ar[d]_{G}&\MHC_{\CC}\ar[d]\\
\ch{Gr\Vect_{\CC}}\ar[r]&\ch{\CC}
}
\]
where the top map is the equivalence of Proposition \ref{equivalenceproof}, $G$ is the associated graded of the weight filtration and the unlabelled maps are the obvious forgetful maps.
This gives us a homotopy commutative square of $\infty$-categories
\[\xymatrix{
\Alg_{\Oo}(\ch{\MHS_\CC})\ar[r]^-{\simeq}\ar[d]_{G}&\Alg_{\Oo}(\MHC_{\CC})\ar[d]\\
\Alg_{\Oo}(\ch{Gr\Vect_{\CC}})\ar[r]&\Alg_{\Oo}(\ch{\CC})
}
\]
(in each case $\Oo$ should be considered as an operad in the relevant $\infty$-category but we have kept the same notation everywhere for simplicity).

We can restrict to $\alpha$-pure objects and we obtain a homotopy commutative square
\[\xymatrix{
\Alg_{\Oo}(\ch{\MHS_\CC}^{\alpha\text{-pure}})\ar[r]^-{\simeq}\ar[d]_{G}&\Alg_{\Oo}(\MHC_{\CC}^{\alpha\text{-pure}})\ar[d]\\
\Alg_{\Oo}(\ch{Gr\Vect_{\CC}}^{\alpha\text{-pure}})\ar[r]&\Alg_{\Oo}(\ch{\CC})
}
\]
From this diagram, we see that it suffices to check that the image of the bottom map consists of formal $\Oo$-algebras. 

By the previous proposition, we have a homotopy commutative triangle of symmetric monoidal $\infty$-categories
\[\xymatrix{
\ch{Gr\Vect_\CC}^{\alpha\text{-pure}}\ar[r]^-{H^*}\ar[dr]&gr\Vect_{\CC}\ar[d]^T\\
 &\ch{\CC}
}
\]
which induces a homotopy commutative square
\[\xymatrix{
\Alg_{\Oo}(\ch{Gr\Vect_\CC}^{\alpha\text{-pure}})\ar[r]^-{H^*}\ar[dr]&\Alg_{\Oo}(gr\Vect_{\CC})\ar[d]^T\\
 &\Alg_{\Oo}(\ch{\CC})
}
\]
Therefore the essential image of the diagonal map consists of formal $\Oo$-algebras as desired.
\end{proof}

\begin{rema}
 The proof of Theorem \ref{purity_implies_formality} is also of course valid for classical mixed Hodge complexes, whose weight filtration exists on a cochain complex defined over a field $\KK\subseteq \RR$. In this case, formality of the underlying algebra is obtained over $\KK$.
\end{rema}

\begin{rema}
One of the main results of \cite{CiHo} gives formality of the forgetful functor 
 \[\MHC_{\CC}^{\alpha\text{-pure}}\lra \ch{\KK}\]
 as a symmetric monoidal functor. This result, although closely related to Theorem \ref{purity_implies_formality} above does not seem to imply it. Indeed, given an $\Oo$-algebra $A$ in pure $\CC$-mixed Hodge complexes, the formality of this functor gives formality 
 of the pair $(\Oo,A)$ in the fibered category of algebras over all operads: assuming that $\Oo$ has trivial differential, this means that there are 
 quasi-isomorphisms of $\Oo_\infty$-algebras $F^*A\leftarrow \bullet\rightarrow G^*H(A)$, where $F,G:\Oo_\infty\to \Oo$ are (possibly distinct) quasi-isomorphisms of operads. Therefore in general, one does not seem to obtain formality of $A$ as an $\Oo$-algebra from the formality of functors.
\end{rema}

\subsection{Mixed Hodge theory of hypercommutative algebras}\label{HycomMHS}

The moduli spaces $\ov{\Mm}_{0,\bullet}$ are smooth complex projective schemes and as such, their homology carries pure Hodge structures.
Therefore we can view $\mathrm{Hycom}$ as an operad in graded $\CC$-mixed Hodge structures which is 1-pure.
The generator $m_i$ of $\mathrm{Hycom}(n)$, which corresponds to the fundamental class $[\ov{\Mm}_{0,n+1}]\in H_{2(n-2)}(\ov{\Mm}_{0,n+1})$, has pure weight $4-2i$, so that:
\[m_n\in W_{4-2n} F^{2-n}\ov {F}{}^{2-n}\mathrm{Hycom}(n)^{4-2n}.\]

In particular, a hypercommutative algebra in $\CC$-mixed Hodge complexes is just a mixed Hodge complex $(A,d,W,F,\ov F)$ together with the structure of a hypercommutative algebra $\{m_n\}$ with the above compatibility condition with respect to the filtrations $W$, $F$ and $\ov F$, so that 
\[\sum_{\sum m_i=m,\sum p_i=p, \sum q_i=q}W_{m_1}F^{p_1}\ov {F}{}^{q_1}A \otimes \cdots\otimes W_{m_n}F^{p_n}\ov {F}{}^{q_n}A\subseteq W_{m+4-2n} F^{p+2-n}\ov {F}{}^{q+2-n}A.\]

The operad $\BV$ can be viewed as an operad in graded mixed Hodge structures. This can be done in an ad-hoc manner by declaring that $\BV_k(n)$ is of pure Hodge type $(-k,-k)$ for any value of $n$. This can be given a geometric interpretation as the homology of a specific model of the framed little $2$-disks operad (see \cite{Vaintrob}). In particular, the arity 1 generator $\Delta$ has pure weight $-2$, so that 
\[\Delta\in W_{-2} F^{-1}\ov{F}{}^{-1}\mathrm{BV}(1)^{-1}.\]

We can also make $\BV/\Delta$ into a $\CC$-mixed Hodge complex by setting the generators $\vp_n$ of degree $-2n$, to be of pure weight $-2n$, so that 
\[\vp_n\in W_{-2n}F^{-n}{\ov{F}}{}^{-n}\BV/\Delta(1)^{-2n}.\]
With this definition, the differential $\delta$ of $\BV/\Delta$ has weight zero and is strictly compatible with respect to all filtrations.
Since $\Delta$ is trivial in homology, $\BV/\Delta$ is an operad in $\MHC_\CC$ which is $1$-pure.
Note as well that the map $\mathrm{Hycom}\to \BV/\Delta$ is compatible with all filtrations and so it is a map of 1-pure $\CC$-mixed Hodge complexes.

The following is now an immediate consequence of Theorem \ref{purity_implies_formality}:

\begin{coro}\label{hycomformal}
 Let $A$ be a hypercommutative algebra in 1-pure $\CC$-mixed Hodge complexes. Then $A$ is formal as a hypercommutative algebra in cochain complexes of $\CC$-vector spaces.
\end{coro}

\begin{rema}
The main geometric situations that we will consider in the following section starts with a $\BV$-algebra which
can be made into a $\CC$-mixed Hodge complex in a compatible way. Since the underlying objects are compact Kähler manifolds, the cohomology is 1-pure (rather than 2-pure). Therefore one does not obtain formality of the $\BV$-algebra, but we will obtain formality of
the associated hypercommutative algebra, whose governing operad is 1-pure.
\end{rema}

\section{Hypercommutative algebra structures on complex manifolds}\label{secGeom}

In this last section we prove the main results of this paper. We consider three geometric situations that lead to $\BV$-algebras whose associated hypercommutative algebras are in fact algebras in $\CC$-mixed Hodge complexes. Using the ``purity implies formality'' theory, we deduce that such hypercommutative structures are formal.

\subsection{Canonical homotopy transfer on a Hermitian manifold}
Let $M$ be a compact complex manifold. Its
 complex de Rham algebra 
$\Aa^*:=\Aa^*_{\mathrm{dR}}(M)\otimes_\RR\CC$ admits a bidegree decomposition by forms of type $(p,q)$:
\[\Aa^n=\bigoplus_{p+q=n} \Aa^{p,q},\]
and the exterior differential decomposes into two components $d=\del+\delb$, where $\del$ has bidegree $(1,0)$ and $\delb$ is its complex conjugate.
Given a Hermitian metric $\langle-,-\rangle$ there is an associated Hodge-star operator 
\[\star:\Aa^{p,q}\lra \Aa^{m-q,m-p}\text{ defined by }\alpha\wedge \star \ov{\beta}:=\langle\alpha,\beta\rangle\mathrm{vol},\]
where  $\mathrm{vol}$
is the volume form determined by the metric and $m$ is the complex dimension of $M$. Let $\delta$ denote either $\del$ or $\delb$. The operator
 \[\delta^*:=-\star \ov{\delta} \star\] 
is the  $\Ll_2$-adjoint  of $\delta$, and we have $d^*=\delb^*+\del^*$.
Consider the associated Laplacian $\Box_\delta:=\delta\delta^*+\delta^*\delta$.
Hodge theory gives orthogonal direct sum decompositions 
\[\Aa^{p,q}=\Hh_\delta^{p,q}\oplus \Box_\delta(\Aa^{p,q})=\Hh_\delta^{p,q}\oplus\Img(\delta)\cap \Aa^{p,q}\oplus \Img(\delta^*)\cap \Aa^{p,q}.\]
where $\Hh_\delta^{p,q}:=\Ker(\Box_\delta)\cap \Aa^{p,q}$ denotes the space of $\delta$-harmonic forms in bidegree $(p,q)$.
Also, we have isomorphisms 
\[\Hh_\delta^{p,q}\cong H_\delta^{p,q}:={\Ker(\delta) \over \Img(\delta)}|_{(p,q)}\]
In particular, any $\delta$-cohomology class admits a unique $\delta$-harmonic representative.
Denote by $\pi_\delta:\Aa^{p,q}\to \Hh_\delta^{p,q}$ the projection and by 
$G_\delta:\Aa^{p,q}\to \Aa^{p,q}$ the Green operator defined by zero on $\delta$-harmonic forms and by the inverse of $\Box_\delta$  on their orthogonal complement. 

We have $G_\delta\Box_\delta=\Box_\delta G_\delta$ 
and any linear operator commuting with $\Box_\delta$ also commutes with $G_\delta$.
In particular, $G_\delta$ commutes with $\delta$ and $\delta^*$.
The following is a matter of verification.

\begin{lemm}\label{httDol} For a compact Hermitian manifold there is a deformation retract
$(\iota,\rho,h)$ for the complex $(\Aa,\delb)$,
where $\iota:H_\delb\to \Aa$ and $\rho:\Aa\to H_\delb$ are given by taking $\delb$-harmonic representatives and projecting to $\delb$-harmonic forms respectively and
 $h:=-\delb^* G_\delb$.
\end{lemm}

Assume now that $M$ is a compact Kähler manifold. 
We will be using the following main properties for the 
differential forms of Kähler manifolds (see for instance \cite{GH}).

\begin{enumerate}
 \item The Laplacian identities $\Box_d=2\Box_\delb=2\Box_\del$ identify all spaces of harmonic forms $\Hh_d^{p,q}=\Hh_\delb^{p,q}=\Hh_\del^{p,q}$ and so all cohomology vector spaces are canonically isomorphic: \[H_d\cong \bigoplus H_\delb^{p,q}\cong \bigoplus H_\del^{p,q}.\]
  \item We have the Kähler identity $\del\delb^*+\delb^*\del=0$.
 \item The double complex $(\Aa,\del,\delb)$ satisfies the $\del\delb$-condition:
 \[\Ker(\del)\cap \Img(\delb)=\Ker(\delb)\cap\Img(\del)=\Img(\del\delb).\]
\end{enumerate}

Lemma \ref{httDol} above gives a canonical deformation retract for the Dolbeault complex of a compact Hermitian manifold. There is an obvious de Rham version of the above result which is valid for any compact Riemannian manifold.
In general, given a deformation retract for the Dolbeault complex of a complex manifold,
one may obtain a deformation retract for the de Rham complex by perturbation of the Dolbeault data (see Theorem 2.7 of \cite{CiSo} and the remark thereafter).
In the Kähler case, such a perturbation turns out to be trivial:

\begin{lemm}\label{httdR}
 For a compact Kähler manifold the data $(\iota,\rho,h)$ of Lemma \ref{httDol} is also a deformation retract for the complex $(\Aa,d)$.
\end{lemm}

\begin{proof}Since $\Hh_\delb=\Hh_\del=\Ker(\del)\cap \Ker(\del^*)$ 
we have $\del\iota=\rho\del=0$. 
This proves that $d\iota=\iota d$ and $d\rho=\rho d$.
Let us prove that $\del h + h\del=0$, where $h=-\delb^* G_\delb$. We have
\[\del\delb^* G_\delb+\delb^* G_\delb\del = 
(\del\delb^* +\delb^* \del) G_\delb=0,
\]
where we used the Kähler identity $\del\delb^*+\delb^*\del=0$ and the fact that, in the Kähler case, the operators $\del$ and $G_\delb$ commute (since any linear operator commuting with $\square_\delb$ also commutes with $G_\delb$, and for Kähler manifolds we have $\square_\delb=\square_\del$).
This gives $dh+hd=\delb h+h\delb=\iota\rho-\Id$.
\end{proof}

\subsection{Kähler manifolds: the $\BV$-algebras of Cao--Zhou}
On any Hermitian manifold, while $\del$ and $\delb$ are derivations with respect to wedge product of differential forms,
their respective formal adjoints  $\del^*$ and $\delb^*$
do not satisfy the Leibniz rule in general.
Instead, on a compact Kähler manifold the tuples 
\[\Aa_{\mathrm{Dol}}:=(\Aa,d=\delb,\wedge,\Delta=-i\del^*)\text{ and }\Aa_{\mathrm{dR}}:=(\Aa,d=\delb+\del,\wedge,\Delta=i(\delb^*-\del^*))\] are $\BV$-algebras, as proven in \cite{CaZh} and \cite{CaZh2}. Note that in both cases the Kähler condition is necessary already for the relation $d\Delta+\Delta d=0$ to be satisfied.
The following is straightforward:

\begin{lemm}Let $\Delta$ denote the BV-operator of either of the $\BV$-algebras $\Aa_{\Dol}$ and $\Aa_{\dR}$.
The deformation retract $(\iota,\rho,h)$ of Lemmas \ref{httDol} and \ref{httdR} satisfies
$\Delta \iota=\rho \Delta=h\Delta+\Delta h=0$.
\end{lemm}

The above lemma together with Proposition \ref{hycombigraded} defines hypercommutative algebra structures on the Dolbeault and de Rham algebras of any compact Kähler manifold. Such structures are canonical, in the sense that they only depend on the Hermitian structure of the manifold (the complex structure together with the chosen compatible Kähler metric).

As is well known, on a  Kähler manifold the commutative algebras $(\Aa,\delb,\wedge)$ and $(\Aa,d,\wedge)$ are quasi-isomorphic and 
formal. In particular, we have a ring isomorphism $H_\delb\cong H_{\dR}$ and the induced $C_\infty$-structure in cohomology is the trivial one.
We promote this result to the hypercommutative case:

\begin{theo}\label{chineseformality}
For any compact Kähler manifold, the canonical hypercommutative structures defined on $\Aa_{\Dol}$ and $\Aa_{\dR}$
are formal and quasi-isomorphic.
\end{theo}
\begin{proof}
We will make $\Aa_{\dR}$ into a hypercommutative algebra in pure $\CC$-mixed Hodge complexes. 
Let $F$ and $\ov{F}$ denote the column and row filtrations respectively:
\[F^p\Aa^n:=\bigoplus_{q\geq p}\Aa^{q,n-q}\quad\text{ and }\quad \ov{F}{}^q\Aa^n:=\bigoplus_{p\geq q}\Aa^{n-p,p}.\]
Since the manifold is Kähler, the double complex $(\Aa,\del,\delb)$ satisfies the $\del\delb$-condition and so by 
Proposition 5.17 of \cite{DGMS} the spectral sequences associated to $F$ and $\ov{F}$ degenerate at $E_1$ and 
induce $n$-opposed filtrations on $H^n$. 
Let $W$ be the canonical filtration
\[
W_p\Aa^n=\left\{ 
\begin{array}{ll}
0&,p<n\\
\Ker(d:\Aa^n\to \Aa^{n+1})&,p=n\\
\Aa^n&,p>n
\end{array}
\right..
\]
Its associated spectral sequence degenerates at $E_1$ and $H^n(Gr_p^W\Aa)\cong H^n(\Aa)$ if $p=n$. 
This proves that the tuple $(\Aa,d,F,\ov{F},W)$ is a pure $\CC$-mixed Hodge complex.
Since 
\[h\Delta=i\delb^* G_\delb (\delb^*-\del^*)=-i \delb^* G_\delb \del^*\]
it follows that the hypercommutative operations $m_n$ on $\Aa_{\dR}$ have bidegree $(2-n,2-n)$. 
This gives \[m_n\in W_{4-2n}\cap F^{2-n}\cap \ov{F}{}^{2-n}\]
which agrees with the mixed Hodge structure on the operad $\Hycom$ given in Section \ref{HycomMHS}. Therefore the operations $m_n$ are compatible with all filtrations,
making $\Aa_{\dR}$ into an hypercommutative algebra in pure $\CC$-mixed Hodge complexes. By Corollary \ref{hycomformal},
it follows that $\Aa_{\dR}$ is a formal hypercommutative algebra.
This is equivalent to the formality of the $\Hycom_\infty$-structure induced on $H_{d}$, which is isomorphic to 
the one induced on $H_\delb$. 
Indeed, both structures are defined using canonical deformation retracts which coincide by Lemma \ref{httdR}.
In particular, $H_\delb$ is a formal $\Hycom_\infty$-algebra, which is equivalent to $\Aa_{\Dol}$ being a formal hypercommutative algebra.
Since both $\Aa_{\dR}$ and $\Aa_{\Dol}$ are formal hypercommutative algebras and the hypercommutative structures on $H_d$ and $H_\delb$ are isomorphic, it follows that $\Aa_{\dR}$ and $\Aa_{\Dol}$ are quasi-isomorphic.
\end{proof}

\begin{rema}\label{morebvs}
 The $\BV$-algebras studied in this section may actually be defined for any Hermitian manifold (not necessarily Kähler).
In the general situation
 one obtains, rather than a $\BV$-algebra, a $\BV_\infty$-algebra with $\BV$-operators $\Delta_k$ given by
\[\Delta_k={1\over k!}\underbrace{[\Lambda,[\Lambda,\cdots[\Lambda}_{k},\delb]]],\]
where $\Lambda$ is the adjoint of the Lefschetz operator.
When the manifold is Kähler, we have
\[\Delta_1=[\Lambda,\delb]=-i\del^*\text{ and }\Delta_k=0,\text{ for }k>1,\]
recovering in this case the $\BV$-algebra of Cao--Zhou.
Moreover, any $\BV_\infty$-algebra defined in this manner is quasi-isomorphic to its underlying commutative algebra.
In this setting, there is a natural solution to the gauge equation of the form $\varphi(z)=\Lambda z$,
which gives trivial hypercommutative structures. Note, however, that the
homotopy hypercommutative algebra structure canonically induced by the hermitian structure, through Lemma \ref{httDol},
is in general non-trivial and non-formal (we refer to \cite{CiWi} for all claims in this remark and further examples. See also \cite{Muro2} for the triviality of hypercommutative structures arising from certain gauge choices).
\end{rema}

We end this section with an example of a non-trivial hypercommutative structure in cohomology.  For this example, we actually consider a non-Kähler orbifold that satisfies the $\del\delb$-condition and such that $\delb\del^*+\del^*\delb=0$. The above theory applies to this broader setting as well.

\begin{exam}
Consider the Iwasawa manifold $M=H_{\ZZ[\sqrt{-1}]}\backslash H_\CC$, defined as the quotient of the complex Heisenberg  group $H_\CC$ by the subgroup of its $\ZZ[\sqrt{-1}]$ points.
Its complex algebra of left-invariant differential forms is isomorphic to the Chevalley-Eilenberg complex of the Lie algebra of $H_\CC$, and so it is given by the free commutative dg-algebra 
\[\Aa=\Lambda(a,b,c,\ov a,\ov b,\ov c)\text{ with }d(c)=-ab, d(\ov c)=-\ov a\ov b.\]
Here $a,b,c$ have bidegree $(1,0)$ and $\ov a,\ov b,\ov c$ have bidegree $(0,1)$. The inclusion of this algebra into the complex de Rham algebra of $M$ preserves bidegrees and induces an isomorphism on Dolbeault cohomology.
The Iwasawa manifold is a non-formal compact complex 3-fold and so it does not satisfy the $\del\delb$-condition. We will however produce a $ \del\delb$-manifold by constructing, out of $M$, an orbifold of global quotient type.
Consider the action $\sigma:\CC^3\to \CC^3$ 
by a finite group of biholomorphisms of order 4, given by $\sigma(z_1,z_2,z_3)=(iz_1,iz_2,-z_3)$.
This descends to a well-defined action on $M$. The resulting quotient 
$\widehat{M}=M/\langle\sigma\rangle$ is a compact orbifold with 16 isolated singular points. 
As shown in Theorem 7.1 of \cite{SfTo}, the orbifold $\widehat{M}$ admits a resolution into a smooth $\del\delb$-compact manifold.
The subalgebra  $\mathcal{A}(\widehat M)$ of $\sigma$-invariant differential forms on $M$ is generated by the elements
\[a\ov a, b\ov b, c\ov c, a\ov b, b\ov a, c\ov a\ov b, ab\ov c,  abc, \ov a\ov b\ov c.\]
The only non-trivial differentials being $\del(c\ov c)=-ab\ov c$, $\del(c\ov a\ov b)=-ab\ov a\ov b$ and their complex conjugates
$\delb(c\ov c)=c\ov a\ov b$, $\delb(a b\ov c )=-ab\ov a\ov b$. One easily checks that this algebra satisfies the 
 $\del\delb$-lemma. In particular, its Hodge-to de Rham spectral sequence degenerates at $E_1$.
We obtain:
\[H_\delb^{*,*}(\widehat{M})\cong 
\arraycolsep=4pt\def\arraystretch{1.4}
 \begin{array}{|c|c|c|c|}
 \hline
[\ov a\ov b\ov c]&0&0&[abc\ov a\ov b\ov c]\\\hline
0&0&[ac\ov a\ov c], [ac\ov b\ov c], [bc\ov b\ov c], [bc\ov a\ov c]&0\\\hline
0&[a\ov a], [b\ov b], [a\ov b], [b\ov a]&0&0\\\hline
1&0&0&[abc]\\\hline
\end{array}
\]
Considering the standard hermitian metric on $\Aa(\widehat{M})$ we obtain operators $\delb^*$ and $\del^*$.
We have 
$\del^*(ab\ov a\ov b)=-c\ov a\ov b$ and $\del^*(ab\ov c)=c\ov c$.
One easily checks that 
$\delb\del^*+\del^*\delb=0$ and that the $d\Delta$-condition is satisfied, with $d=\delb$ and $\Delta=-i\del^*$. Therefore the tuple $(\Aa(\widehat{M}), d=\delb,\wedge,\Delta=-i\del^*)$ is a $\BV$-algebra (since $d\Delta+\Delta d=0$)
and we may use Proposition \ref{hycombigraded} to compute
$\varphi_1=h\Delta=i\delb^* G_{\delb}\del^*$. We obtain
$\varphi_1(a b \ov a\ov b)=i c\ov c$ and $\varphi_1=0$ otherwise.
Using the formula of Example \ref{theta3} we compute a non-trivial operation of arity 3 and bidegree $(-1,-1)$ on cohomology: 
\[
m_3([a\ov a], [b\ov b], [b\ov b])=-2[\varphi_1(a\ov ab\ov b)\cdot b\ov b]=-2i[bc\ov b\ov c].
\]
\end{exam}

\begin{rema}
The operation of arity 3 computed in the above example is in fact a Gerstenhaber Massey product in the sense of Proposition 5.2 in \cite{Muro}.
In particular, this shows that,
while we have formality as a commutative, as a Lie, and as a hypercommutative algebra (by Theorem \ref{chineseformality}),
 the underlying Gernstenhaber algebra is not formal,
which in turn implies non-formality as a $\BV$-algebra. Note that this example does not fit in the construction explained in Remark \ref{morebvs}, for which conjugation by $\Lambda$ would give extra terms on the $\BV$-operator.
\end{rema}

\subsection{Calabi--Yau manifolds: the $\BV$-algebra of Barannikov-Kontsevich}\label{BVBK}
We recall the $\BV$-algebra considered by Barannikov and Kontsevich in \cite{BaKo} (see also Section III.9.12 of \cite{Maninbook}).
Let $M$ be a compact Calabi--Yau manifold of complex dimension $m$. This is a compact Kähler manifold $M$ with trivial canonical bundle. In particular, we may choose a nowhere vanishing holomorphic $m$-form $\Omega$.
Consider the complex bigraded algebra $(\Ll^{*,*},\wedge)$ given by 
\[\Ll^{p,q}:=\Gamma(M,\Lambda^p T\otimes \Lambda^q\ov{T}^*),\]
where $T$ denotes the holomorphic tangent bundle of $M$. The algebra structure is determined by 
the
exterior product $\wedge$.
There is an isomorphism of bigraded vector spaces $\eta:\Ll^{p,q}\to \Aa^{m-p,q}$ given by the contraction $\alpha\mapsto \alpha \vdash \Omega$
with the holomorphic $m$-form $\Omega$. Note that such an isomorphism does not preserve the product structures.
Given an operator $\delta:\Aa^{*,*}\to \Aa^{*+p,*+q}$ of pure bidegree $(p,q)$, we will denote 
\[\delta_\eta:=\eta^{-1}\circ\delta\circ \eta:\Ll^{*,*}\to \Ll^{*-p,*+q}.\]
the operator of bidegree $(-p,q)$ induced on $\Ll$ via the isomorphism $\eta$.
On the other hand, the Hermitian metric gives an isomorpism $T\cong \ov{T}^*$ inducing, for each $(p,q)$, a $\CC$-linear isomorphism
$g:\Ll^{p,q}\to \Ll^{q,p}$.

\begin{lemm}[c.f. \cite{CaZh}, Lemma 3.1]\label{involutedifferentials}The following identities are satisfied:
\[\arraycolsep=4pt\def\arraystretch{1.4}
\begin{array}{ll}
\delb_\eta=\Id\otimes\delb \,\,,&\del^*_\eta=-g\circ (\Id\otimes\delb)\circ g=-g \circ \delb_\eta\circ g
\\ \delb^*_\eta=\Id\otimes\delb^*,&
\del_\eta=-g\circ (\Id\otimes\delb^*)\circ g=-g \circ \delb^*_\eta \circ g.
\end{array}
\]
\end{lemm}
\begin{proof}
The two identities on the left are straightforward. We prove the identity involving $\del_\eta$. 
Since both the Kähler form $\omega$ and the holomorphic $m$-form $\Omega$ are parallel with respect to the Levi-Civita connection,
on an open subset $U\subseteq M$
there is a local orthonormal frame field  $\{e_1,\cdots,e_m\}$ of the holomorphic tangent bundle $T=T^{1,0}M$ (so that 
for all $x\in U$, $\{e_1(x),\cdots,e_m(x)\}$ is an orthonormal basis of $T_p^{1,0}M$) such that, up to a constant,
$\Omega=e^1\wedge\cdots\wedge e^m$, and $\omega=\sum e^i\wedge \ov{e}^i$, where
 $\{e^1,\cdots,e^m\}$ denotes the dual coframe, defined by $e^j(e_i)=\delta^j_i$. 
Also, at the point $x$, we have
\[\nabla_{e_i}e_j=\nabla_{\ov{e}_i}e_j=\nabla_{e_i}\ov{e}_j=\nabla_{\ov{e}_i}\ov{e}_j=0\]
A standard computation shows that the operators $\del$, $\delb$ and their adjoints may be written locally as:
\[
\delb =\ov{e}^i\wedge \nabla_{\ov{e}_i}\quad,\quad\del^* =-{e}_i\vdash \nabla_{\ov{e}_i}\quad,\quad
\del ={e}^i\wedge \nabla_{{e}_i}\quad\text{ and }\quad\delb^* =-\ov{e}_i\vdash \nabla_{{e}_i},
\]
where $(X\vdash \alpha)(X_1,\cdots,X_{k-1}):=\alpha(X,X_1,\cdots,X_{k-1})$ is the contraction of a vector field $X$ and a $k$-form $\alpha$
(see for instance Lemma 3.3.4 of \cite{Jost}).

Near $x$ any element in $\Ll^{p,q}$ may be written as a sum of elements of the form $\alpha=fe_I\otimes \ov{e}^{J}$, where $I$ and $J$ denote subsets of $\{1,\cdots,m\}$ of cardinalities $p$ and $q$ respectively. We have 
\[
g(\Id\otimes \delb^*)g \alpha=g(\Id\otimes \delb^*)(fe_J\otimes \ov{e}^I)=
-(\nabla_{e_i}f) g (e_J\otimes \ov{e}_i\vdash \ov{e}^I)=-(\nabla_{e_i}f)({e}^i\vdash {e}_I)\otimes \ov{e}^J.
\]
On the other hand, we have 
\[
\del_\eta(\alpha)=\eta^{-1}\del \eta(\alpha)=\eta^{-1} \del(f(e_I\vdash \Omega)\otimes\ov{e}^{J})=\eta^{-1}(e^i\wedge \nabla_{e_i}f (e_I\vdash \Omega)\otimes \ov{e}^J=(\nabla_{e_i}f)({e}^i\vdash {e}_I)\otimes \ov{e}^J.
\]
The identity involving $\del^*_\eta$ follows analogously.

\end{proof}

Barannikov and Kontsevich noted in \cite{BaKo} that the complex $(\Ll^{*,*},\delb_\eta)$ carries a $\BV$-algebra structure
with the wedge product $\wedge$ and where the $\Delta$ operator is given by $\del_\eta$.
Note that $\delb_\eta$ has bidegree $(0,1)$ while $\del_\eta$ has bidegree $(-1,0)$. 
We next define a ``de Rham version'' of this $\BV$-algebra, in the sense that we add, to the differential and the $\Delta$ operator,
a component of bidegree $(1,0)$ and $(0,-1)$ respectively, both defined using the original operators together with the involution $g:\Ll^{p,q}\to \Ll^{q,p}$.

\begin{prop}The following tuples are all $\BV$-algebras:
\[\arraycolsep=4pt\def\arraystretch{1.4}
\begin{array}{ll}
\Ll_{\Dol}:=(\Ll^{*,*},\delb_\eta,\wedge,\del_\eta),\\
\Ll_{\Dol^*}:=(\Ll^{*,*},\del^*_\eta,\wedge,-\delb^*_\eta),\\
\Ll_{\dR}:=(\Ll^{*,*},d:=\delb_\eta+\del^*_\eta,\wedge,\Delta:=\del_\eta-\delb^*_\eta).
\end{array}
\]
\end{prop}
\begin{proof}
The first is the $\BV$-algebra of \cite{BaKo}. It is clear that $(\Ll^{*,*},\del^*_\eta,-\delb^*_\eta)$ is a double complex. Moreover,
by Lemma \ref{involutedifferentials} we have
\[
\del^*_\eta=-g \circ \delb_\eta\circ g
\quad\text{ and }\quad
\del_\eta=-g \circ \delb^*_\eta \circ g,
\]
where $g:\Ll^{p,q}\to \Ll^{q,p}$ denotes the involution induced by $T\cong \ov{T}^*$. Since this involution preserves algebra structures and $\Ll_{\Dol}$ is a $\BV$-algebra, it follows that the tuple $(\Ll^{*,*},-\delb^*_\eta,\wedge,-\del^*_\eta)$ is a $\BV$-algebra and so
$\Ll_{\mathrm{Dol^*}}=(\Ll^{*,*},\delb^*_\eta,\wedge,-\del^*_\eta)$ is also a $\BV$-algebra.
We now show that $\Ll_{\dR}$ is a $\BV$-algebra.
The conditions $d^2=\Delta^2=0$ are straightforward. Also, we have 
\[d\Delta+\Delta d=\eta^{-1}(\Box_\del-\Box_\delb)\eta=0.\] Note that for this last property, the minus sign in $\Delta=\del_\eta-\delb^*_\eta$ is important.
Therefore $(\Ll,d,\Delta)$ is a double complex. It now suffices to note that given $\BV$-algebras $(A,d_i,\wedge,\Delta_i)$ for $i=0,1$ then $(A,d_0+d_1,\wedge, \Delta_0+\Delta_1)$ is a $\BV$-algebra whenever $(A,d_0+d_1,\Delta_0+\Delta_1)$ is a double complex.
\end{proof}

In analogy with Lemma \ref{httdR} we have:

\begin{lemm}\label{httBK}The deformation retract $(\iota,\rho,h)$ of Lemma \ref{httDol} induces a deformation retract
$(\iota_\eta,\rho_\eta,h_\eta)$ for both complexes $(\Ll,\delb_\eta)$ and $(\Ll,\delb_\eta+\del^*_\eta)$.
Moreover, we have \[
\Delta \iota_\eta=\rho_\eta \Delta=h_\eta\Delta+\Delta h_\eta=0,\] where 
$\Delta$ denotes the BV-operator of either $\Ll_{\Dol}$ or $\Ll_{\dR}$.
\end{lemm}

The above lemma together with Proposition \ref{hycombigraded} defines canonical hypercommutative algebra structures on 
$\Ll_{\Dol}$ and $\Ll_{\dR}$ respectively. The main result of this paper is the following:

\begin{theo}\label{mainthm}
For any compact Calabi--Yau manifold, the canonical hypercommutative structures defined on $\Ll_{\Dol}$ and $\Ll_{\dR}$ are formal and quasi-isomorphic.
\end{theo}
\begin{proof}
The proof is analogous to that of Theorem \ref{chineseformality}.
We first make $\Ll_{\dR}$ into a hypercommutative algebra in $\CC$-mixed Hodge complexes that is pure.
Note first that the components of the differential $d=\delb_\eta+\del^*_\eta$ of $\Ll^{*,*}$ have bidegrees $(0,1)$ and $(1,0)$ respectively.
Let $F$ and $\ov{F}$ be the column and row filtrations of $\Ll^{*,*}$ respectively.
The double complex $(\Ll,\delb_\eta,\del^*_\eta)$ satisfies the $\delb_\eta\del^*_\eta$-condition.
This follows from the fact that the Laplacians 
\[\Box_{\delb_\eta}:=\delb_\eta\delb_\eta^*+\delb_\eta^*\delb_\eta=(\Box_\delb)_\eta\quad\text{ and }\quad
 \Box_{\del_\eta}:=\del_\eta\del_\eta^*+\del_\eta^*\del_\eta=(\Box_\del)_\eta
\]
coincide, together with
the Kähler identity
$\delb\del^*+\del^*\delb=0$, which makes $\delb_\eta$ and $\del^*_\eta$ anticommute.
Therefore by \cite[Proposition 5.17]{DGMS} the spectral sequences associated to $F$ and $\ov{F}$ degenerate at $E_1$ and 
induce $n$-opposed filtrations on $H^n_d(\Ll)$, where $d=\delb_\eta+\del^*_\eta$.
Let $W$ be the canonical filtration. 
Its associated spectral sequence degenerates at $E_1$ and $H^n(Gr_p^W\Ll)\cong H^n(\Ll)$ if $p=n$. 
This proves that the tuple $(\Ll,d,F,\ov{F},W)$ is a pure $\CC$-mixed Hodge complex.
The bidegree of $h_\eta \Delta$ is $(-1,-1)$. Indeed, we have 
\[h_\eta\Delta=-\eta^{-1}(i\delb^* G_\delb)\eta (\del_\eta-\delb^*_\eta)=
-i \eta^{-1} (\delb^*\del G_\delb -\delb^*\delb^* G_\delb)\eta
=-i \eta^{-1} \delb^*\del G_\delb \eta.\]
It follows that the hypercommutative operations $m_n$ have bidegree $(2-n,2-n)$. 
This gives \[m_n\in W_{4-2n}\cap F^{2-n}\cap \ov{F}{}^{2-n}\]
which agrees with the mixed Hodge structure on the operad $\Hycom$ of Section \ref{HycomMHS}. This makes $\Ll_{\dR}$ into a hypercommutative algebra in pure $\CC$-mixed Hodge complexes, and so $\Ll_{\dR}$ is a formal hypercommutative algebra by Corollary 
\ref{hycomformal}.
This proves formality of $\Ll_{\dR}$.
The proof now follows exactly as in Theorem \ref{chineseformality}.
\end{proof}

\subsection{Hermitian manifolds: when $\Delta$ is a derivation}
We end this paper by considering a version of a $\BV$-algebra $(\Aa,d,\cdot,\Delta)$ for which $\Delta$ is a derivation and so its associated Lie bracket is trivial.
Such algebras are called $\BV_1$-algebras in \cite{DST}, where $\Delta$ is understood as an operator of order one.
In this case, there is a quasi-isomorphism of operads $\Hycom_1\to \BV_1/\Delta$,
where  $\Hycom_1$ is generated by elements $\nu_i$ of
 arity 2 and degree $|\nu_i|=-2i$ satisfying
\[\sum_{i+j=k} \nu_i\circ_1\nu_j=\sum_{i+j=k} \nu_i\circ_2\nu_j.\]
In analogy with Example \ref{theta3}, the lowest terms are $\theta_0(x,y)=x\cdot y$ and
\[\theta_1(x,y)=\varphi_1(x\cdot y)-\varphi_1(x)\cdot y-x\cdot \varphi_1(y).\]

The following proposition shows that, under the $d\Delta$-condition, the $\Hycom_1$-algebra
induced by a deformation retract is trivial, in the sense that it is quasi-isomorphic to the underlying commutative algebra.
\begin{prop}\label{Geofftrivial}
 Assume that a $\BV_1$-algebra $(\Aa,d,\cdot,\Delta)$ satisfies the $d\Delta$-condition. Then:
 \begin{enumerate}[(1)]
  \item $(\Aa,d,\wedge,\Delta)$ is formal as a $\BV$-algebra.
  \item For any deformation retract $(\iota,\rho,h)$  satisfying
$\Delta \iota=\rho \Delta=h\Delta+\Delta h=0,$
  the induced $\Hycom_1$-algebra structure is formal and quasi-isomorphic to the commutative algebra $(H,\wedge)$ \end{enumerate}
\end{prop}
\begin{proof}
 By Lemma \ref{quises} the $d\Delta$-condition gives quasi-isomorphisms of $\BV$-algebras 
\[(\Aa,d,\Delta)\hookleftarrow (\Ker\Delta,d,0)\twoheadrightarrow(H,0,0).\]
By Lemma \ref{HTData} there is a deformation retract $(\iota,\rho,h)$ such that
\[\Delta \iota=\rho\Delta=h\Delta+\Delta h=0,\]
giving a hypercommutative algebra structure on $(\Aa,d,\cdot)$. The above zig-zag of quasi-isomorphisms extends to a zig-zag of deformation retracts and so it gives trivial data on the complex $(H,0)$.
\end{proof}

\begin{exam}
The Dolbeault algebra $(\Aa,\delb,\wedge)$ of any compact Kähler manifold, with the operator $\Delta=\del$, fits in the setting of Proposition \ref{Geofftrivial}. In fact, only the $\del\delb$-condition is needed, so the proposition applies as well to any
compact complex manifold which can be blown-up to a Kähler manifold.
\end{exam}

Assuming the weaker condition of Hodge-to-de-Rham degeneration data $(\iota,\rho,h)$ on a $\BV_1$-algebra,
there is still an induced $\Hycom_1$-algebra structure.
This structure is not formal in general, as shown in the following example.

\begin{exam}
The Kodaira-Thurston manifold is 
the $4$-dimensional compact nilmanifold given by the quotient
$\mathrm{KT}:=H_\ZZ \times \ZZ \setminus H_\RR \times \RR$
where $H_\RR$ is the Heisenberg real Lie group of dimension 3 and $H_\ZZ$ is the integral subgroup.
This manifold admits a left-invariant complex structure, making it into a non-Kähler primary Kodaira surface.
Its complex homotopy type may be described as follows:
consider the free bigraded commutative algebra
\[\Aa:=\Lambda(a,\ov a,b,\ov b)\quad\text{ with }|a|=|b|=(1,0)\text{ and }|\ov a|=|\ov b|=(0,1).\]
Define $\delb$ on $\Aa$ by $\delb b=-ia\ov a$ and extending multiplicatively and let $\del$ be defined by complex conjugation.
There is an inclusion of $\Aa$ into the complex de Rham algebra of $\mathrm{KT}$ which preserves
the bidegrees and induces an isomorphism in Dolbeault cohomology (see Example 4.3 of \cite{CiSo} for details).
We obtain:
\[H_\delb^{*,*}(KT)\cong 
\arraycolsep=4pt\def\arraystretch{1.4}
 \begin{array}{|c|c|c|c|}
 \hline
[\ov a \ov b]&[b \ov a \ov b ]&[ab\ov a\ov b] \\
  \hline
 [\ov a], [\ov b]&[a\ov b], [b \ov a ]&[ab\ov a ], [ab\ov b] \\
   \hline
 1&[a]&[ab]\\ 
 \hline
\end{array}
\]
Since $\mathrm{KT}$ is a compact complex surface, its Frölicher spectral sequence degenerates at the first page, and so the above gives
 the complex de Rham cohomology of $\mathrm{KT}$. Considering the standard hermitian metric on $\Aa$ 
the deformation retract $(\iota,\rho,h)$ of Lemma \ref{httDol} gives
\[h(a\ov a):=-\delb^* G_\delb(a\ov a)=*\delb *(a\ov a)=* \delb(b\ov b)=i*a\ov a\ov b=ib,\]
\[h(a\ov a \ov b):=-\delb^* G_\delb(a\ov a  \ov b)=*\delb *(a\ov a  \ov b)=* \delb(b)=i*a\ov a=ib\ov b,\]
 and $h=0$ otherwise.
Moreover, this data satisfies $\rho\Delta\iota=0$ and $\Delta h \Delta \iota=0$ and so the Hodge-to-de-Rham degeneration condition is satisfied.
Letting $\varphi_1=h\Delta-\iota\rho\Delta h$, we have
\[\varphi_1(\ov b)=b,\, \varphi_1(a\ov a\ov b)=ab\ov a\text{ and }\varphi_1=0\text{ otherwise},\]
while $\varphi_n=0$ for all $n>1$. Using the formula
\[\nu_1(x,y)=\varphi_1(x\cdot y)-\varphi_1(x)\cdot y-x\cdot \varphi_1(y)\]
we get non-trivial $\Hycom_1$-operations in cohomology:
\[\nu_1([\ov a],[\ov b])=[b\ov a]\text{ and }\nu_1([a],[\ov b])=-[ab].\]
There are also non-trivial $C_\infty$-operations of arity 3 and bidegree $(0,-1)$:
\[\mu_{3}([a],[a],[\ov a])=i[ab]\quad\text{ and }\quad\mu_3([a],[\ov a],[\ov a])=-i[b\ov a].\]
This shows that the $\Hycom_1$-structure is not formal and carries more information than the usual $C_\infty$-structure.
\end{exam}

\bibliographystyle{alpha}
\bibliography{biblio}

\end{document}